\documentclass{amsart}
\usepackage{pdfpages}
\usepackage{hyperref}
\newtheorem{theorem}{Theorem}[section]

\newtheorem{corollary}[theorem]{Corollary}
\theoremstyle{definition}
\newtheorem{definition}[theorem]{Definition}

\theoremstyle{remark}

\numberwithin{equation}{section}

\begin{document}

\title{Conformal image of an osculating curve on a smooth immersed surface}

\author{Absos Ali Shaikh$^1$, Mohamd Saleem Lone$^2$ and Pinaki Ranjan Ghosh$^{3}$}
\address{$^1$Department of Mathematics, University of
Burdwan, Golapbag, Burdwan-713104, West Bengal, India}
\email{aask2003@yahoo.co.in, aashaikh@math.buruniv.ac.in}
\address{$^2$International Centre for Theoretical Sciences, Tata Institute of Fundamental Research, 560089, Bengaluru, India}
\email{saleemraja2008@gmail.com, mohamdsaleem.lone@icts.res.in}
\address{$^3$Department of Mathematics, University of
Burdwan,Golapbag, Burdwan-713104, West Bengal, India}
\email{mailtopinaki94@gmail.com}






\begin{abstract}
The main intention of the paper is to investigate an osculating curve under the conformal map. We obtain a sufficient condition for the conformal invariance of an osculating curve. We also find an equivalent system of a geodesic curve under the conformal transformation(motion) and show its invariance under isometry and homothetic motion. 
\end{abstract}

\maketitle
\section{Introduction} 
The study of smooth maps between surfaces is an important field of study in differential geometry. There are various mappings which preserve certain differential geometric quantities. Depending on the invariance of the mean curvature and the Gaussian curvature, we mainly classify the transformations as isometric, conformal and non-conformal. In case of isometry both length as well as the angle is preserved. A geometrical way of interpreting isometry is that it preserves the Gaussian curvature but not the mean curvature. One of the examples of such a transformation is the existence of an isometry between a helicoid and a catenoid. A  generalized class of isometry is called a conformal transformation which is a dilated form of isometry. In the case of conformal motion, only angles are preserved and not necessarily distances. An important example of conformal transformation is a stereographic projection.  Gerardus Mercator used the conformal property of stereographic projection to develop the famous Mercator's world map. It is believed  the world's first angle preserving map. A beautiful explanation together with some applications of conformal maps, Bobenko and Gunn recently published a movie with Springer on conformal maps, we strongly appeal the reader to see \cite{1}. A more generalized class of motions is of non-conformal transformations, wherein neither distances nor angles are preserved.

The geometric position of an arbitrary point on a curve or on a surface is affirmed by the position vector field. In case of curve, the position vector field can be thought of a trajectory of the curve, wherein its first derivative gives the velocity and the second derivative gives the acceleration of the trajectory. In this paper, we are going to study the conformal properties of a curve which totally depends upon its position vector field. If the position vector field of a curve always lies in the orthogonal complement of the binormal vector, we say that type of curve as an osculating curve. In other words an osculating curve is a curve whose position vector always lies in the osculating plane spanned by unit tangent vector and unit normal vector. Similarly we can define normal curve as a curve whose position vector lies in the plane spanned by the unit normal and binormal vector. Chen (\cite{2}) introduced rectifying curves and obtained a relation between a rectifying curve and the ratio of its curvature and torsion. Chen (\cite{2}) characterized the rectifying curve as the ratio of its curvature and torsion is a non-constant linear function of the parameter. For details we refer the reder to see \cite{3,4}.

The motivation of this paper starts with the article \cite{10}. In \cite{10} Shaikh and Ghosh investigated the invariant properties of a rectifying curve lying on a smooth immersed surface under isometric transformation. In that paper, they obtained a sufficient condition with respect to which a rectifying curve retains its rectifying property under an isometric motion. In addition to this sufficient condition, they proved that the normal component of the rectifying curve is preserved under the same isometric motion. This paper laid a foundation for many problems, for example, what happens to an osculating or a normal curve lying on smooth immersed surfaces under the  isometric motion? Both of these problems were discussed in \cite{8,9,12}. Later in \cite{11}, we generalize the notion of study by conformal transformation, wherein we study the conformal transformation of rectifying curves lying on smooth surfaces. In this paper(\cite{11}), in addition to various geometric invariants, we obtain a sufficient condition with respect to which a rectifying curve retains its nature under conformal motion. Now, the first hand possible studies depending upon the position vector field and the conformality will be the investigation of osculating and normal curves. The present paper is devoted to investigate osculating curves and the question with respect to the normal curve can be a future problem.

The paper is framed as follows. Section $2$ is devoted to some rudimentary facts  about the curves lying on a smooth immersed surface. In section $3$, we discuss the main results. We obtain a sufficient condition for the conformal invariance of an osculating curve. We also find a condition for the homothetic invariance of the normal component of the osculating curve and prove that the component of an osculating curve along any tangent vector to the surface is homothetic invariant. We also find an equivalent system of a geodesic curve under conformal transformation and show its invariance under isometry and homothetic motion. 
\section{Preliminaries} 
Suppose $\vec{t}$, $\vec{n}$ and $\vec{b}$ are respectively the tangent, normal and binormal vectors at any point to a unit speed curve $\sigma:I\subset \mathbb{R}\rightarrow \mathbb{E}^3$ such that $\{\vec{t},\ \vec{n},\ \vec{b}\}$ forms a Serret-Frenet frame. If $^\prime$ denotes the differentiation with respect to the arc length parameter $s$, then we have 
\begin{eqnarray*}
\left\{
\begin{array}{ll}
\frac{d\vec{t}}{ds} =\kappa \vec{n}\\
\frac{d\vec{n}}{ds} = -\kappa \vec{t} +\tau \vec{b}\\
\frac{d\vec{b}}{ds} =-\tau \vec{n},
\end{array}
\right.
\end{eqnarray*}
where $\kappa$ and $\tau$ are respectively the curvature and torsion of the curve $\sigma$.

\begin{definition}
A diffeomorphism $\mathcal{G}:\mathcal{M}\rightarrow\tilde{\mathcal{M}} $ is said to be local isometry between the surfaces $\mathcal M$ and $\tilde{\mathcal M}$, if for all $x_1,x_2 \in T_p({\mathcal M})$ and for any $p\in {\mathcal M}$, $$\langle x_1, x_2 \rangle_p =\langle d\mathcal{G}_p(x_1),d\mathcal{G}_p(x_2) \rangle_{\mathcal{G}(p)}.$$ A diffeomorphism which in addition to local isometry is a bijective map is called an isometry. If there exists such an isometry $\mathcal{G}: \mathcal M \rightarrow \tilde{\mathcal M},$ then $\mathcal M$ and $\tilde{\mathcal M}$ are said to be isometric. 

If $\mathcal{G}:{\mathcal M}\rightarrow \tilde{{\mathcal M}}$ is a local isometry, then the coefficients of their first fundamental forms are invariant under $\mathcal{G}$ and hence, 

\begin{equation*}
E=\tilde{E},\quad F=\tilde{F},\quad G=\tilde{G}.
\end{equation*}
\end{definition}
\begin{definition} 
Let $\mathcal{G}$ be a diffeomorphism between two smooth surfaces $\mathcal M$ and $\tilde{\mathcal M}$. Then $\mathcal{G}$ is said to be a local conformal map between $\mathcal M$ and $\tilde{\mathcal M}$, if for all $x_1,x_2 \in T_p({\mathcal M})$ with an arbitrary $p\in {\mathcal M}$, we have
$$\delta^2\langle d\mathcal{G}_p(x_1),d\mathcal{G}_p(x_2) \rangle_{\mathcal{G}(p)}=\langle x_1, x_2 \rangle_p, $$ where $\delta^2$ is a differentiable function on ${\mathcal M}$, also known as dilation factor.  We see that a conformal motion is the composition of an isometry and a dilation. We observe that if the dilation factor is identity, then we get the isometry. Geometrically, we can say that the conformal maps preserve angles both in direction and magnitude but not necessarily the lengths. In this case \cite{5}: 
\begin{equation*}
\delta^2E=\tilde{E},\quad  \delta^2F=\tilde{F},\quad  \delta^2G=\tilde{G}.
\end{equation*}
Here we shall call that the first fundamental form coefficients are conformally invariant.
\end{definition}
 For such a $\mathcal{G}$ to be conformal, a necessary and sufficient condition is that their line elements are proportional and the ratio of arc elements given by $\frac{ds}{d\tilde{s}}$ is equal to the dilation factor $\delta$. If for all points on the surface the dilation factor is a non-zero constant(say $c$), then the conformal map is said to be homothetic. If the dilation function is constantly equal to one, the conformal map becomes an isometry. Thus we can say isometric maps are a subset of conformal maps with the dilation factor $\delta=1$ \cite{6}.

\begin{definition}\label{def1}
Let $\mathcal M$ and  $\tilde{\mathcal M}$ be two smooth surfaces with $\Phi(u,v)$ being the surface patch of $\mathcal M$ and let $g:{\mathcal M}\rightarrow \tilde{\mathcal M}$ be a smooth map such that $\tilde{g}= g \circ \Phi$:
\begin{itemize}
\item If $g$ is conformal, then $g$ is said to be conformally invariant when  $\tilde{g}=\delta^2g$ for some dilation factor $\delta(u,v).$ 
\item If $g$ is homothetic, then $g$ is said to be homothetic invariant when  $\tilde{g}=c^2g, (c\neq \{0,1\})$. 
\end{itemize}
\end{definition}

\begin{definition}
Let $\sigma : I \subset \mathbb{R} \rightarrow \mathbb{E}^3$ be a smooth curve. Then
 $\sigma$ is said to be an osculating curve if its position vector is lying in the orthogonal complement of normal vector i.e., $\sigma \cdot \vec{n} =0,$ and hence
\begin{equation}\label{1}
\sigma(s)=\xi(s)\vec{t}(s)+ \mu(s)\vec{n}(s),
\end{equation}
where $\xi,$ $\mu$ are two smooth functions. 
\end{definition}

Let ${\mathcal{M}}$ be a smooth surface with 
$\Phi(u,v):U\subset\mathbb{R}^2\rightarrow \mathcal{M}$ being its chart map ([page no 52, \cite{5}]). Then, for the curve 
$\sigma(s)=\Phi(u(s),v(s))$ on $\mathcal{M}$, we have
\begin{eqnarray}
\nonumber\sigma^\prime(s)&=&\Phi_uu^\prime+\Phi_vv^\prime,\\
\nonumber\text{or, }&&\\
\label{2} \vec{t}(s)&=&\sigma^\prime(s)=\Phi_uu^\prime+\Phi_vv^\prime\\
\nonumber
 {\vec{t}'}(s)&=& u^{\prime\prime}\Phi_u+v^{\prime\prime}\Phi_v+{u^\prime}^2\Phi_{uu}+2u^\prime v^\prime \Phi_{uv}+{v^\prime}^2\Phi_{vv}.
\end{eqnarray}
If ${\bf N}$ is the surface normal, then
\begin{equation}\label{3}
\vec{n}(s)=\frac{1}{\kappa(s)}\vec{t}'(s)=\frac{1}{k(s)}(\Phi_uu''+\Phi_vv''+\Phi_{uu}u'^2+\Phi_{uv}2u'v'+\Phi_{vv}v'^2).\end{equation}
\begin{eqnarray}\nonumber
\vec{b}(s)&=&\frac{1}{k(s)}\Big[(\Phi_uu'+\Phi_vv')\times(\Phi_uu''+\Phi_vv''+\Phi_{uu}u'^2+\Phi_{uv}2u'v'+\Phi_{vv}v'^2)\Big],\\
\nonumber&=&\frac{1}{k(s)}\Big[\{u'v''-u''v'\}{\bf N}+\Phi_u\times \Phi_{uu}u'^3+2\Phi_u\times \Phi_{uv}u'^2v'+\Phi_u\times \Phi_{vv}u'v'^2\\
\nonumber &&+\Phi_v\times \Phi_{uu}u'^2v'+2\Phi_v\times \Phi_{uv}u'v'^2+\Phi_v\times \Phi_{vv}v'^3\Big].
\end{eqnarray}
\begin{definition}
If $\sigma$ be a unit speed curve on ${\mathcal{M}}$, then $\vec{t}\perp{\bf N}$
and hence ${\bf N} \times \sigma^\prime$, $\sigma^\prime$ and ${\bf N}$ are mutually orthogonal to each other such that $\sigma^{\prime\prime}=\kappa_g{\bf N}\times \sigma^\prime+\kappa_{n}{\bf N}$, where $\kappa_g$ and $\kappa_n$ are respectively known as the geodesic and normal curvature of $\sigma$.
Again since $\sigma^{\prime \prime}=\kappa(s)\vec{n}(s)$, we have
\begin{equation*}
\kappa_n=\kappa(s)\vec{n}(s)\cdot {\bf N}=(u''\Phi_u+v''\Phi_v+u'^2\Phi_{uu}+2u'v'\Phi_{uv}+v'^2\Phi_{vv})\cdot{\bf N}
\end{equation*}
or 
\begin{equation}\label{se1}
\kappa_n = {u^\prime}^2 L + 2u^\prime v^\prime M +{v^\prime}^2N,
\end{equation}
where $L,M,N$ are the coefficients of the second fundamental form of the surface. We say that $\sigma$ is asymptotic iff $\kappa_n=0.$
\end{definition}
\section{Conformal image of an osculating curve.}
Let $\sigma(s)$ be an osculating curve lying on a smooth immersed surface ${\mathcal{M}}$ in $\mathbb{E}^3$, then with the help of (\ref{1}), (\ref{2}) and (\ref{3}), we get
\begin{equation}\label{w2.1}
\sigma(s)=\xi(s)(\Phi_uu'+\Phi_vv')+\frac{\mu(s)}{k(s)}\Big[u''\Phi_u+v''\Phi_v+u'^2\Phi_{uu}+2u'v'\Phi_{uv}+v'^2\Phi_{vv}\Big].
\end{equation}
\begin{theorem}
If $\mathcal{G}:{\mathcal M}\rightarrow \tilde{{\mathcal M}}$ is a conformal map, then the image $\tilde{\sigma}(s)$ of $\sigma(s)$ under $\mathcal{G}$ is an osculating curve on $\tilde{{\mathcal M}}$ when
\begin{eqnarray}\label{j2.2}
\nonumber
\tilde{\sigma}-\delta \mathcal{G}_\ast(\sigma)&=&\frac{\mu}{\kappa}\Big[{u^\prime}^2 \left(\delta_u \mathcal{G}_\ast \Phi_u +  \delta\frac{\partial \mathcal{G}_*}{\partial u}\Phi_u\right)+2{u^\prime} v^\prime \left(\delta_u \mathcal{G}_\ast \Phi_v +  \delta\frac{\partial \mathcal{G}_*}{\partial u}\Phi_v\right)\\
&&+{v^\prime}^2 \left(\delta_v \mathcal{G}_\ast \Phi_v +  \delta\frac{\partial \mathcal{G}_*}{\partial v}\Phi_v\right)\Big].
\end{eqnarray}
\end{theorem}
\begin{proof}
Let $\tilde{\mathcal{M}}$ be the conformal image of ${\mathcal{M}}$ and $\Phi(u,v)$ and $\tilde{\Phi}(u,v)=\mathcal{G}\circ \Phi(u,v) $ be the surface patches of ${\mathcal{M}}$ and $\tilde{{\mathcal{M}}},$ respectively. Then the differential map $d\mathcal{G}=\mathcal{G}_\ast$ of $\mathcal{G}$ sends each vector of the tangent space $T_p{\mathcal{M}}$ to a dilated tangent vector of the tangent space of $T_{\mathcal{G}(p)}\tilde{\mathcal{M}}$ with the dilation factor $\delta$. 
\begin{eqnarray}\label{2.2}
\tilde{\Phi}_u(u,v)&=&\delta(u,v) \mathcal{G}_*\Phi_u,\\
\label{2.3}\tilde{\Phi}_v(u,v)&=&\delta(u,v) \mathcal{G}_*\Phi_v.
\end{eqnarray}
Differentiating $(2.2)$ and $(2.3)$ partially with respect to both $u$ and $v$ respectively, we get
\begin{eqnarray}\label{q3.7}
\nonumber \tilde{\Phi}_{uu}&=& \delta_u \mathcal{G}_\ast \Phi_u +  \delta\frac{\partial \mathcal{G}_*}{\partial u}\Phi_u+\delta \mathcal{G}_*\Phi_{uu}\\
\tilde{\Phi}_{vv}&=&\delta_v \mathcal{G}_\ast \Phi_v+ \delta\frac{\partial \mathcal{G}_*}{\partial v}\Phi_v+\delta \mathcal{G}_*\Phi_{vv}\\
\nonumber \tilde{\Phi}_{uv}&=& \delta_u \mathcal{G}_\ast \Phi_v+ \delta\frac{\partial \mathcal{G}_*}{\partial u}\Phi_v+\delta \mathcal{G}_*\Phi_{uv}\\
\nonumber &=& \delta_v \mathcal{G}_\ast \Phi_u+ \delta\frac{\partial \mathcal{G}_*}{\partial v}\Phi_u +\delta \mathcal{G}_*\Phi_{uv}.\\\nonumber
 \end{eqnarray}

 Therefore in view of (\ref{j2.2}), (\ref{2.2}), (\ref{2.3}) and (\ref{q3.7}), we obtain
\begin{eqnarray*}
\tilde{\sigma}&=&\xi(u^\prime \delta \mathcal{G}_\ast \Phi_u+v^\prime \delta \mathcal{G}_\ast \Phi_v)+\frac{\mu}{\kappa}\Big[{u^\prime}^2 \left(\delta_u \mathcal{G}_\ast \Phi_u +  \delta\frac{\partial \mathcal{G}_*}{\partial u}\Phi_u+\delta \mathcal{G}_\ast \Phi_{uu}\right)\\
&&+2{u^\prime} v^\prime \left(\delta_u \mathcal{G}_\ast \Phi_v +  \delta\frac{\partial \mathcal{G}_*}{\partial u}\Phi_v+\delta \mathcal{G}_\ast \Phi_{uv}\right)+{v^\prime}^2 \left(\delta_v \mathcal{G}_\ast \Phi_v +  \delta\frac{\partial \mathcal{G}_*}{\partial v}\Phi_v+\delta \mathcal{G}_\ast \Phi_{vv}\right)\Big],
\end{eqnarray*}
which can be written as
\begin{eqnarray*}
\nonumber
\tilde{\sigma}(s)&=&\xi(s)(\tilde{\Phi}_u u'+\tilde{\Phi}_v v')+\frac{\mu(s)}{k(s)}\Big[u''\tilde{\Phi}_u+v''\tilde{\Phi}_v+u'^2\tilde{\Phi}_{uu}+2u'v'\tilde{\Phi}_{uv}+v'^2\tilde{\Phi}_{vv}\Big],
\end{eqnarray*}
i.e.,
\begin{equation*}
\tilde{\sigma}(s)=\tilde{\xi}(s)\tilde{\vec{t}}(s)+\frac{\tilde{\mu}(s)}{\tilde{\kappa}(s)}\tilde{\vec{n}}(s)
\end{equation*}
for some $C^\infty$ functions $\tilde{\xi}(s)$ and $\tilde{\mu}(s).$ Thus $\tilde{\sigma}(s)$ is an osculating curve.
\end{proof}
Here and now on onward, we accept 
$\tilde{\xi}$, $\tilde{\mu}$ and $\tilde{\kappa}$ are also dilated with $\delta$ in such a way, such that 
$\tilde{\mu}/ \tilde{\kappa}=\mu / \kappa.$
\begin{corollary}
Let ${\mathcal{M}}$ and ${\tilde{\mathcal{M}}}$ be two smooth surfaces and $\mathcal{G}$ be a homothetic map between them and $\sigma(s)$ be an osculating curve on ${\mathcal{M}}$. Then $\tilde{\sigma}(s)$ is an osculating curve on $\tilde{\mathcal{M}}$ if
\begin{eqnarray*}
\nonumber
\tilde{\sigma}-c \mathcal{G}_\ast(\sigma)&=&\frac{\mu}{\kappa}\Big[{u^\prime}^2 \left(  c\frac{\partial \mathcal{G}_*}{\partial u}\Phi_u\right)+2{u^\prime} v^\prime \left(  c\frac{\partial \mathcal{G}_*}{\partial u}\Phi_v\right)
+{v^\prime}^2 \left(  c\frac{\partial \mathcal{G}_*}{\partial v}\Phi_v\right)\Big].
\end{eqnarray*}
\end{corollary}
\begin{corollary}\cite{10}
Let $\mathcal{G}:{\mathcal{M}}\rightarrow {\tilde{\mathcal{M}}}$ be an isometry between two smooth immersed surfaces ${\mathcal{M}}$ and $\tilde{\mathcal{M}}$ in $\mathbb{E}^3$ and $\sigma(s)$ be an osculating curve on ${\mathcal{M}}$. Then $\tilde{\sigma}(s)$ is an osculating curve on $\tilde{\mathcal{M}}$ if
\begin{eqnarray*}
\nonumber
\tilde{\sigma}- \mathcal{G}_\ast(\sigma)&=&\frac{\mu}{\kappa}\Big[{u^\prime}^2   \frac{\partial \mathcal{G}_*}{\partial u}\Phi_u+2{u^\prime} v^\prime \frac{\partial \mathcal{G}_*}{\partial u}\Phi_v
+{v^\prime}^2 \frac{\partial \mathcal{G}_*}{\partial v}\Phi_v \Big].
\end{eqnarray*}
\end{corollary}
\begin{theorem}
If the smooth surfaces ${\mathcal{M}}$ and $\tilde{\mathcal{M}}$ are conformally related and $\sigma(s)$ is a non-asymptotic osculating curve on ${\mathcal{M}}$, then for the component of $\sigma$ along the surface normal, the following relation holds:
\begin{equation*}
\tilde{\sigma}\cdot{\tilde{\bf N}}-\delta^2(\sigma\cdot {\bf N})=\frac{\mu}{\kappa}\frac{1}{\delta^4 W^2}\Big[{u^\prime}^2(\tilde{\Phi}_{uu}-\delta^6 \Phi_{uu})+2u^\prime v^\prime (\tilde{\Phi}_{uv}-\delta^6 \Phi_{uv})+{v^\prime}^2(\tilde{\Phi}_{vv}-\delta^6 \Phi_{vv})\Big].
\end{equation*}
\end{theorem}
\begin{proof}
Let $\sigma(s)$ be an osculating curve on $\mathcal{M}$ whose at least second order partial derivatives are non-vanishing. To find the position vector of $\sigma$ along the normal ${\bf N}$ to the surface ${\mathcal{M}}$ at a point $\sigma(s)$, from (\ref{w2.1}), we have
\begin{equation}\label{k3.7}
\sigma(s)\cdot {\bf N} =\frac{\mu(s)}{k(s)}\Big[u'^2L+2u'v'M+v'^2N\Big],
\end{equation}
where $L,$ $M$ and $N$ are the second fundamental form coefficients. In the Monge patch form, the coefficients are given by
\begin{equation}\label{q20}
L=\frac{\Phi_{uu}}{1+\Phi_u^2+\Phi_v^2(=W^2)},\quad M=\frac{\Phi_{uv}}{1+\Phi_u^2+\Phi_v^2},\quad N=\frac{\Phi_{vv}}{1+\Phi_u^2+\Phi_v^2}. 
\end{equation}
Now if $\tilde{\sigma}(a)$ is an osculating curve on $\tilde{M}$, we have
\begin{equation}\label{k3.9}
\tilde{\sigma}\cdot \tilde{\bf N}= \frac{\mu(s)}{k(s)}\Big[u'^2\tilde{L}+2u'v'\tilde{M}+v'^2\tilde{N}\Big].
\end{equation}
Again by the Monge patch form, the conformal coefficients are given by
\begin{equation}\label{q20}
\tilde{L}=\frac{\tilde{\Phi}_{uu}}{1+\tilde{\Phi}_u^2+\tilde{\Phi}_v^2},\quad \tilde{M}=\frac{\tilde{\Phi}_{uv}}{1+\tilde{\Phi}_u^2+\tilde{\Phi}_v^2},\quad \tilde{N}=\frac{\tilde{\Phi}_{vv}}{1+\tilde{\Phi}_u^2+\tilde{\Phi}_v^2}. 
\end{equation}
From (\ref{k3.7}) and (\ref{k3.9}), we get
\begin{equation*}
\tilde{\sigma}\cdot{\tilde{\bf N}}-\delta^2 (\sigma \cdot {\bf N})=\frac{\mu}{\kappa}\Big[{u^\prime}^2(\tilde{L}-\delta^2 L)+2u^\prime v^\prime (\tilde{M}-\delta^2 M)+{v^\prime}^2(\tilde{N}-\delta^2 N)\Big].
\end{equation*}
Using (\ref{q20}) the above equation turns out to be
\begin{equation}\label{q21}
\tilde{\sigma}\cdot{\tilde{\bf N}}-\delta^2(\sigma\cdot {\bf N})=\frac{\mu}{\kappa}\frac{1}{\delta^4 W^2}\Big[{u^\prime}^2(\tilde{\Phi}_{uu}-\delta^6 \Phi_{uu})+2u^\prime v^\prime (\tilde{\Phi}_{uv}-\delta^6 \Phi_{uv})+{v^\prime}^2(\tilde{\Phi}_{vv}-\delta^6 \Phi_{vv})\Big],
\end{equation}
which proves the theorem.
\end{proof}
\begin{corollary}\label{cor3.5}
If ${\mathcal{M}}$ and $\tilde{\mathcal{M}}$ are conformally related such that $\sigma(s)$ is osculating curve on ${\mathcal{M}}$, then the normal component of $\sigma(s)$ is invariant under the conformal map if any one of the following relation holds:
\begin{itemize}
\item[(i)] The position vector of $\sigma(s)$ is in the direction of the tangent vector.
\item[(ii)] The curve $\sigma(s)$ is asymptotic.
\item[(iii)] The normal curvature is invariant under the confomal map.
\end{itemize}
\end{corollary}
\begin{proof}
From (\ref{se1}), (\ref{k3.7}) and (\ref{k3.9}), we get
\begin{equation}
\tilde{\sigma}\cdot \tilde{\bf N}-\delta^2(\sigma \cdot {\bf N})= \frac{\mu(\tilde{\kappa}_n -\delta^2 \kappa_n)}{\kappa}.
\end{equation}
Now $\tilde{\sigma}\cdot {\bf N}-\delta^2(\sigma \cdot {\bf N})=0$ if and only if $\mu =0$ and $\tilde{\kappa}_n -\delta^2 \kappa_n\neq 0,$ implying $\sigma(s)=\xi(s)t(s)$ which proves $(i)$ of the corollary. $(ii)$ and $(iii)$ are the direct implications of $\tilde{\kappa}_n -\delta^2 \kappa_n=0$.  
\end{proof}
\begin{corollary}
If ${\mathcal{M}}$ and $\tilde{\mathcal{M}}$ are homothetic and $\sigma(s)$ is an osculating curve on ${\mathcal{M}}$, then the normal component of $\sigma(s)$ is homothetically invariant if any one of the following relation holds:

\begin{itemize}
\item[(i)] $(i)$ and $(ii)$ of corollary $\ref{cor3.5}$ holds.
\item[(ii)] The normal curvature is homothetic invariant.
\end{itemize}
\end{corollary}
\begin{corollary}
If ${\mathcal{M}}$ and $\tilde{\mathcal{M}}$ are isometric and $\sigma(s)$ is an osculating curve on ${\mathcal{M}}$, then the component of $\sigma$ along the surface normal is invariant under isometry if any one of the following relation holds:
\begin{itemize}
\item[(i)] $(i)$ and $(ii)$ of corollary $\ref{cor3.5}$ holds.
\item[(ii)] The normal curvature is invariant under isometry.
\end{itemize}
\end{corollary}
\begin{theorem}
If $\mathcal{G}:{\mathcal M}\rightarrow \tilde{{\mathcal M}}$ is a conformal map and $\sigma(s)$ is an osculating curve on ${\mathcal M}$, then for the tangential component we have
\begin{eqnarray}\label{de2}
\tilde{\sigma}(s)\cdot {\bf T}= \delta^2 \left(\sigma \cdot {\bf T}\right)+h(E,F,G,\delta),
\end{eqnarray}
where
\begin{eqnarray}\label{de1}h(E,F,G,\delta)&=&
\nonumber \frac{\mu}{2\kappa}\Big[a\left(2{u^\prime}^2 \delta \delta_u E+4u^\prime v^\prime \delta \delta_v E+4{v^\prime}^2 \delta \delta_v F-2{v^\prime}^2 \delta \delta_u G\right)\\
&&+ b\left(4{u^\prime}^2\delta \delta_u F-2{u^\prime}^2 \delta \delta _v E+4{v^\prime}u^\prime \delta \delta_u G+2{v^\prime}^2 \delta \delta_v G\right)\Big].
\end{eqnarray}
\end{theorem}
\begin{proof}
Let $\tilde{\mathcal{M}}$ be the conformal image of ${\mathcal{M}}$ and $\Phi(u,v)$ and $\tilde{\Phi}(u,v)=\mathcal{G}\circ \Phi(u,v) $ be the surface patches of ${\mathcal{M}}$ and $\tilde{{\mathcal{M}}},$ respectively. We know that 
\begin{equation}\label{z2.8}
\delta^2E=\tilde{E},\quad  \delta^2F=\tilde{F},\quad  \delta^2G=\tilde{G}.
\end{equation}
This implies that
\begin{eqnarray}\label{12a}
\left\{
\begin{array}{ll}
\tilde{E}_u=2\delta \delta_u E + \delta^2 E_u, \quad \tilde{E}_v=2\delta \delta_v E + \delta^2 E_v,\\
\tilde{F}_u=2\delta \delta_u F + \delta^2 F_u, \quad \tilde{F}_v=2\delta \delta_v F + \delta^2 F_v,\\
\tilde{G}_u=2\delta \delta_u G + \delta^2 G_u, \quad \tilde{G}_v=2\delta \delta_v G + \delta^2 G_v.
\end{array}
\right.
\end{eqnarray}
Now, we have
$$E_u=(\Phi_u \cdot \Phi_u)_u=2\Phi_{uu}\cdot \Phi_u$$
\begin{equation}\label{a2.9}
\Phi_{uu}\cdot \Phi_u=\frac{E_u}{2}.
\end{equation}
Similarly, it is easy to check that
\begin{eqnarray}\label{a2.10}
\begin{array}{ll}
\Phi_{uu}\cdot \Phi_v=F_u-\frac{E_v}{2}, \quad \Phi_{uv}\cdot \Phi_u=\frac{E_v}{2}, \quad \Phi_{uv}\cdot \Phi_v=\frac{G_u}{2},\\
\Phi_{vv}\cdot \Phi_v=\frac{G_v}{2}, \quad \Phi_{vv}\cdot \Phi_u=F_v-\frac{G_u}{2}.
\end{array}
\end{eqnarray}
Thus from (\ref{w2.1}), (\ref{a2.9}) and (\ref{a2.10}), we can easily deduce
\begin{equation}
\sigma(s)\cdot \Phi_u = \xi(s)(u^\prime E + v^\prime F)+\frac{\mu(s)}{2\kappa(s)}[2u^{\prime \prime}E+2v^{\prime \prime} F +{u^\prime}^2E_u +2u^\prime v^\prime E_v +2{v^\prime}^2F_v-{v^\prime}^2G_u].
\end{equation}
Similarly
\begin{equation}
\sigma(s)\cdot \Phi_v = \xi(s)(u^\prime F + v^\prime G)+\frac{\mu(s)}{2\kappa(s)}[2u^{\prime \prime}F+2v^{\prime \prime} G +2{u^\prime}^2F_u -{u^\prime}^2 E_v +2{v^\prime}u^\prime G_u+{v^\prime}^2G_v].
\end{equation}
Now if $\tilde{\sigma}$ be an osculating curve on $\tilde{\mathcal{M}}$ and ${\bf T}=a\Phi_u + b\Phi_v$ be the tangent vector of $\tilde{\mathcal{M}}$ at $\tilde{\sigma}(s)$, we have
\begin{eqnarray*}
\tilde{\sigma}(s)\cdot (a\tilde{\Phi}_u +b \tilde{\Phi}_v)&=&a\Big[\tilde{\xi}(s)(u^\prime \tilde{E} + v^\prime \tilde{F})+\frac{\tilde{\mu}(s)}{2\tilde{\kappa}(s)}\Big(2u^{\prime \prime}\tilde{E}+2v^{\prime \prime} \tilde{F} +{u^\prime}^2\tilde{E}_u\\
&& +2u^\prime v^\prime \tilde{E}_v +2{v^\prime}^2\tilde{F}_v-{v^\prime}^2\tilde{G}_u\Big)\Big]+b\Big[\tilde{\xi}(s)(u^\prime \tilde{F} + v^\prime \tilde{G})\\
&&+\frac{\tilde{\mu}(s)}{2\tilde{\kappa}(s)}\Big(2u^{\prime \prime}\tilde{F}+2v^{\prime \prime} \tilde{G} +2{u^\prime}^2\tilde{F}_u -{u^\prime}^2 \tilde{E}_v +2{v^\prime}u^\prime \tilde{G}_u+{v^\prime}^2\tilde{G}_v\Big)\Big],
\end{eqnarray*}
or
\begin{eqnarray*}
\tilde{\sigma}(s)\cdot {\bf T}-\delta^2 [\sigma(s)\cdot {\bf T}]&=&
\frac{\mu}{2\kappa}\Big[a\left(2{u^\prime}^2 \delta \delta_u E+4u^\prime v^\prime \delta \delta_v E+4{v^\prime}^2 \delta \delta_v F-2{v^\prime}^2 \delta \delta_u G\right)\\
&&+ b\left(4{u^\prime}^2\delta \delta_u F-2{u^\prime}^2 \delta \delta _v E+4{v^\prime}u^\prime \delta \delta_u G+2{v^\prime}^2 \delta \delta_v G\right)\Big].
\end{eqnarray*}
\end{proof}
\begin{corollary}
Let $\mathcal{G}:{\mathcal{M}}\rightarrow {\tilde{\mathcal{M}}}$ be a homothetic conformal map between two smooth surfaces and $\sigma(s)$ be an osculating curve on $\mathcal{M}$. Then the tangential component is homothetic invariant.
\end{corollary}
\begin{proof}
The claim directly follows from (\ref{de2}) and (\ref{de1}) while assuming $\delta(u,v)=c$. 
\end{proof}
\begin{corollary}\cite{10}
Let $\mathcal{G}:{\mathcal{M}}\rightarrow {\tilde{\mathcal{M}}}$ be an isometry  and $\sigma(s)$ be an osculating curve on $\mathcal{M}$. Then the tangential component of $\sigma$ remains invariant.
\end{corollary}
\begin{theorem}
Let $\mathcal{G}:{\mathcal{M}}\rightarrow {\tilde{\mathcal{M}}}$ be a conformal map between two smooth surfaces and $\sigma(s)$ be an osculating curve on $\mathcal{M}$. Then for the geodesic curvature, we have
\begin{equation}\label{q23}
\tilde{\kappa}_g=\delta^2\kappa_g+f(E,F,G,\delta),
\end{equation}
where 
\begin{equation*}
f(E,F,G,\delta)=\Big[\epsilon_{11}^2{u^\prime}^3+(2\epsilon_{12}^2-\epsilon_{11}^1){u^\prime}^2 v^\prime +(\epsilon_{22}^2-2\epsilon_{12}^1)u^\prime {v^\prime}^2-\epsilon_{22}^1{v^\prime}^3\Big] \sqrt{EG-F^2}.
\end{equation*}
\end{theorem}
\begin{proof}
Let $\sigma(s)$ be an osculating curve on $\mathcal{M}$ and $\tilde{\sigma}=\mathcal{G}\circ \sigma$ be the conformal image of $\sigma$ on $\tilde{\mathcal{M}}$. From the definition of geodesic curvature
\begin{equation*}
\kappa_g=\sigma^{\prime \prime}\cdot ({\bf N}\times \sigma^\prime).
\end{equation*}
In \cite{8} Shaikh and Ghosh showed that for such a $\sigma$, we have
\begin{eqnarray*}
\kappa_g &=& u^\prime v^{\prime \prime}(EF-FE)+v^\prime u^{\prime \prime}(F^2- GE)+u^\prime v^{\prime \prime}(EG-F^2)+v^\prime v^{\prime \prime}(FG-GF)\\
&&+{u^\prime}^3(E\Phi_{uu}\cdot \Phi_v -F \Phi_{uu}\cdot \Phi_u)+{u^\prime}^2v^\prime (F\Phi_{uu}\cdot \Phi_v -G\Phi_{uu}\cdot \Phi_u)\\
&&+2{u^\prime}^2v^\prime(E\Phi_{uv}\cdot \Phi_v-F\Phi_{uv}\cdot \Phi_u)+2u^\prime{v^\prime}^2(F\Phi_{uv}\cdot \Phi_v-G\Phi_{uv}\cdot \Phi_u)\\
&&+ u^\prime{v^\prime}^2(E\Phi_{vv}\cdot \Phi_v - F \Phi_{vv}\cdot \Phi_u)+{v^\prime}^3(F\Phi_{vv}\cdot \Phi_v -G \Phi_{vv}\cdot \Phi_{u}).
\end{eqnarray*}
With the help of (\ref{a2.9}) and (\ref{a2.10}), the above equation turns out to be
\begin{eqnarray}
\nonumber \kappa_g&=&(u^\prime v^{\prime\prime}-v^\prime u^{\prime\prime})(EG-F^2)+\frac{1}{2}{u^\prime}^3(2EF_u-EE_v-FE_u)\\
\nonumber &&+\frac{1}{2}{u^\prime}^2v^\prime(2FF_u-FE_v-GE_u)+{u^\prime}^2v^\prime(EG_u-FE_v)+u^\prime{v^\prime
}^2(FG_u-GE_v)\\
\label{21k}&&+\frac{1}{2}u^\prime{v^\prime}^2(EG_v-2FF_v+FG_u)+\frac{1}{2}{v^\prime}^3(FG_v-2GF_v+GG_u).
\end{eqnarray}
In addition, let $\Gamma_{ij}^k$ be the Christoffel symbols of second kind given by
\begin{equation}\label{q12}\left\{
\begin{array}{ll}
\Gamma_{11}^1=\frac{1}{2W^2}\left\{GE_u+F[E_v-2F_u]\right\}, \quad \Gamma_{22}^2=\frac{1}{2W^2}\left\{EG_v+F[G_v-2F_v]\right\}\\
\Gamma_{11}^2=\frac{1}{2W^2}\left\{E[2F_u-E_v]-FE_v\right\},\quad \Gamma_{22}^1=\frac{1}{2W^2}\left\{G[2F_v-G_u]-FG_v\right\}\\
\Gamma_{12}^2=\frac{1}{2W^2}\left\{EG_u-FE_v\right\}=\Gamma_{21}^2,\quad
\Gamma_{21}^1=\frac{1}{2W^2}\left\{GE_v-FG_u\right\}=\Gamma_{12}^1,
\end{array}
\right.
\end{equation}
where $W=\sqrt{EG-F^2}$.
After conformal motion, the Christoffel symbols turns out to be
\begin{eqnarray}\label{q13}
\begin{array}{ll}
\tilde{\Gamma}_{11}^1=\Gamma_{11}^1 +\vartheta_{11}^1,\quad \tilde{\Gamma}_{11}^2=\Gamma_{11}^2 +\vartheta_{11}^2, \quad \tilde{\Gamma}_{12}^1=\Gamma_{12}^1 +\vartheta_{12}^1,\\
\tilde{\Gamma}_{12}^2=\Gamma_{12}^2 +\vartheta_{12}^2,\quad \tilde{\Gamma}_{22}^1=\Gamma_{22}^1 +\vartheta_{22}^1, \quad \tilde{\Gamma}_{22}^2=\Gamma_{22}^2 +\vartheta_{22}^2,
\end{array}
\end{eqnarray}
where
\begin{eqnarray}\label{q16}
\left\{\begin{array}{ll}
\vartheta_{11}^1=
\frac{EG\delta_u -2F^2\delta_u +FE\delta_v}{\delta W^2},\quad \vartheta_{11}^2=
\frac{EF\delta_u -E^2 \delta_v}{\delta W^2},\vspace{.1cm}\\
\vartheta_{12}^1=\frac{EG\delta_v-FG\delta_u}{\delta W^2},
\quad \vartheta_{12}^2=\frac{EG\delta_u - FE \delta_v}{\delta W^2},\vspace{.1cm}\\
\vartheta_{22}^1=\frac{GF\delta_v - G^2 \delta_u}{\delta W^2},\quad \vartheta_{22}^2=\frac{EG\delta_v -2F^2 \delta_v +FG\delta_u}{\delta W^2}.
\end{array}\right.
\end{eqnarray}
Using (\ref{q12}) in (\ref{21k}), we get
\begin{equation}\label{w23}
\kappa_g=\Big[\Gamma_{11}^2{u^\prime}^3+(2\Gamma_{12}^2-\Gamma_{11}^1){u^\prime}^2 v^\prime +(\Gamma_{22}^2-2\Gamma_{12}^1)u^\prime {v^\prime}^2-\Gamma_{22}^1{v^\prime}^3+u^\prime v^{\prime\prime}-u^{\prime \prime}v^\prime\Big] \sqrt{EG-F^2}.
\end{equation}
In view of (\ref{q13}) and the above equation, $\tilde{\kappa}_g$ is given by
\begin{equation}\label{q23}
\tilde{\kappa}_g=\delta^2\kappa_g+\Big[\epsilon_{11}^2{u^\prime}^3+(2\epsilon_{12}^2-\epsilon_{11}^1){u^\prime}^2 v^\prime +(\epsilon_{22}^2-2\epsilon_{12}^1)u^\prime {v^\prime}^2-\epsilon_{22}^1{v^\prime}^3\Big] \sqrt{EG-F^2}.
\end{equation}
This proves the claim.
\end{proof}
\begin{corollary}
Let $\mathcal{G}$ be homothetic conformal map between two smooth surfaces $\mathcal{M}$ and $\tilde{\mathcal{M}}$. Then the geodesic curvature of rectifying curve is homothetic invariant under $\mathcal{G}$.
\end{corollary}
\begin{proof}
Let us suppose $\delta=c,$ then the proof is a direct implication of (\ref{q16}) and (\ref{q23}).
\end{proof} 
\begin{corollary}\cite{8}
Let $\mathcal{G}$ be an isometry between two smooth surfaces $\mathcal{M}$ and $\tilde{\mathcal{M}}$. Then the geodesic curvature of rectifying curve is invariant under $\mathcal{G}$.
\end{corollary}
\section{Conformal image of Geodesics}
In this section, we seek what happens to a geodesic, in particular an osculating curve on a smooth immersed surface under conformal transformation.  
\begin{definition}A vector field $X$ is said to be parallel along a curve $\sigma:I\rightarrow {\mathcal{M}}$ if $\frac{DX}{dx}=0$ for all $x\in I.$
\end{definition}
\begin{definition}
A non-constant curve $\sigma:I \rightarrow {\mathcal{M}}$ is said to be geodesic at $x\in I$ if its tangent vector field is parallel at $x\in I$, i.e., $\frac{D \sigma^\prime(x)}{dx}=0$ and $\sigma$ is said to be a geodesic on $\mathcal{M}$ if it is geodesic for all $x\in {\mathcal{M}}$.
\end{definition}
\begin{theorem}\label{thm4.3}
Let $\mathcal{G}:{\mathcal{M}}\rightarrow \tilde{\mathcal{M}}$ be a conformal map between two smooth surfaces $\mathcal{M}$ and ${\tilde{\mathcal{M}}}$ and $\sigma(s)$ be a geodesic on $\mathcal{M}$. Then $\tilde{\sigma}=\mathcal{G} \circ \sigma$ being a geodesic on $\tilde{\mathcal{M}}$ is equivalent to the following system of differential equations
\begin{equation}\label{conf1}
\left\{
\begin{array}{ll}
u^{\prime\prime}+ \Gamma_{11}^1{u^\prime}^2+2\Gamma_{12}^1{u^\prime}v^\prime +\Gamma_{22}^1{v^\prime}^2+f_1(E,F,G,\delta)=0\\
v^{\prime\prime}+ \Gamma_{11}^2{u^\prime}^2+2\Gamma_{12}^2{u^\prime}v^\prime +\Gamma_{22}^2{v^\prime}^2+f_2(E,F,G,\delta)=0.
\end{array}
\right.
\end{equation}
\end{theorem}
\begin{proof}
Let $\sigma(s)= {\Phi}(u(s),v(s))$ be a parameterized geodesic on the surface ${\mathcal{M}}$ and $\Phi(u,v)$ be the coordinate chart of $\mathcal{M}$. Then, we have
\begin{equation*}
t(s)=\sigma^\prime(s)=\Phi_uu^\prime+\Phi_vv^\prime.
\end{equation*}
Taking the covariant derivative of the above expression, the parallel condition of the tangent vector field of $\sigma$ is obtained as
\begin{equation}\label{r4}
\Big(u^{\prime\prime}+ \Gamma_{11}^1{u^\prime}^2+2\Gamma_{12}^1{u^\prime}v^\prime +\Gamma_{22}^1{v^\prime}^2\Big)\Phi_u + \Big(v^{\prime\prime}+ \Gamma_{11}^2{u^\prime}^2+2\Gamma_{12}^2{u^\prime}v^\prime +\Gamma_{22}^2{v^\prime}^2\Big)\Phi_v=0,
\end{equation}where $\Gamma_{ij}^k,(i,j,k=1,2)$ are Christoffel symbols given by (\ref{q12}). 

\noindent Since $\Phi_u$ and $\Phi_v$ are two basis vectors, therefore the geodesic condition in (\ref{r4}) is equivalent to the following system of differential equations:
\begin{equation}\label{b2}
\left\{
\begin{array}{ll}
u^{\prime\prime}+ \Gamma_{11}^1{u^\prime}^2+2\Gamma_{12}^1{u^\prime}v^\prime +\Gamma_{22}^1{v^\prime}^2=0\\
v^{\prime\prime}+ \Gamma_{11}^2{u^\prime}^2+2\Gamma_{12}^2{u^\prime}v^\prime +\Gamma_{22}^2{v^\prime}^2=0.
\end{array}
\right.
\end{equation}
\noindent Let $\tilde{\sigma}$ be a conformal image of $\sigma$ on $\tilde{\mathcal{M}}$. Using (\ref{q13}) and (\ref{q16}), we have
\begin{equation}
\left\{
\begin{array}{ll}
u^{\prime\prime}+ \Gamma_{11}^1{u^\prime}^2+2\Gamma_{12}^1{u^\prime}v^\prime +\Gamma_{22}^1{v^\prime}^2+\Big( \vartheta_{11}^1{u^\prime}^2+2\vartheta_{12}^1{u^\prime}v^\prime +\vartheta_{22}^1{v^\prime}^2\Big)=0\\
v^{\prime\prime}+ \Gamma_{11}^2{u^\prime}^2+2\Gamma_{12}^2{u^\prime}v^\prime +\Gamma_{22}^2{v^\prime}^2+\Big( \vartheta_{11}^2{u^\prime}^2+2\vartheta_{12}^2{u^\prime}v^\prime +\vartheta_{22}^2{v^\prime}^2\Big)=0,
\end{array}
\right.
\end{equation}
where $\epsilon_{ij}^k,(i,j,k=1,2)$ are given by (\ref{q16}).
Therefore the fact that $\tilde{\sigma}$ is a geodesic on $\tilde{\mathcal{M}}$ is equivalent to the following system of differential equations:
\begin{equation*}
\left\{
\begin{array}{ll}
u^{\prime\prime}+ \Gamma_{11}^1{u^\prime}^2+2\Gamma_{12}^1{u^\prime}v^\prime +\Gamma_{22}^1{v^\prime}^2+f_1(E,F,G,\delta)=0\\
v^{\prime\prime}+ \Gamma_{11}^2{u^\prime}^2+2\Gamma_{12}^2{u^\prime}v^\prime +\Gamma_{22}^2{v^\prime}^2+f_2(E,F,G,\delta)=0,
\end{array}
\right.
\end{equation*}
$f_1(E,F,G,\delta)=\vartheta_{11}^1{u^\prime}^2+2\vartheta_{12}^1{u^\prime}v^\prime +\vartheta_{22}^1{v^\prime}^2$ and $f_2(E,F,G,\delta)=\vartheta_{11}^2{u^\prime}^2+2\vartheta_{12}^2{u^\prime}v^\prime +\vartheta_{22}^2{v^\prime}^2$. This proves the result.
\end{proof}
\begin{corollary}\label{cor4.4}  A geodesic say $\sigma(s)$ on a smooth surface is invariant under isometry and homothetic motion.
\end{corollary}
\begin{proof}Noting that from (\ref{q13}) and (\ref{q16}), the Christofell symbols are invariant under isometry of homothetic motion. From (\ref{conf1}), we see that the geodesic conditions are dilated with the dilation factors $f_1$ and $f_2$, but from (\ref{conf1}), (\ref{q13}) and (\ref{q16}), it is easy to judge that the geodesics remain invariant under isometry and homothetic motion.
\end{proof}
{\bf Note:} The conclusions in theorem \ref{thm4.3} and corollary \ref{cor4.4} are true for a general space curve, the same is true, in particular if $\sigma(s)$ is an osculating curve.

\section{acknowledgment}
 The second author greatly acknowledges to The University Grants Commission, Government of India for the award of Junior Research Fellow.

\end{document}